\documentclass[a4paper,english,oneside]{smfcustom}
\usepackage{mathrsfs,mathtools}
\usepackage[english]{babel}
\usepackage[utf8]{inputenc}
\usepackage[T1]{fontenc}
\usepackage{pxfonts,microtype,enumitem}
\usepackage{xcolor}\usepackage[all]{xy}
\usepackage[labelfont=normal]{subfig}
\usepackage{tikz}
\usetikzlibrary{patterns,arrows}
  \def\O{\mathcal{O}} \def\Z{\mathbb{Z}}
\def\P{\mathbf{P}}
\def\F{\mathbf{F}}
\def\G{\mathbf{G}}
 
\def\coeff{{\propto}}
\newcommand{\dual}[1]{#1^{\mathrm{c}}}
\def\rev{\raisebox{.15ex}{$\scriptscriptstyle \leftarrow$}}
\def\bydef{\coloneqq}
\DeclareMathOperator{\rk}{rank}

\DeclarePairedDelimiter{\abs}{\lvert}{\rvert} 
\DeclarePairedDelimiter{\Set}{\{}{\}} 
\def\tprod{\mathop{\textstyle\prod}}

\def\Wedge^#1{\mathop{\operatorname{\resizebox{.8em}{!}{$\wedge$}}\!\!^{#1}}}
\def\skew#1#2{\left.#1\middle/#2\right.}
\NoSwapTheoremNumbers
\numberwithin{equation}{section}
\author[L.~Darondeau]{Lionel Darondeau}
\address{Departement Wiskunde, KU Leuven}
\email{lionel.darondeau@normalesup.org}
\author[P.~Pragacz]{Piotr Pragacz}
\address{Institute of Mathematics, Polish Academy of Sciences}
\email{P.Pragacz@impan.pl}
\title
[Gysin maps, duality and Schubert classes]
{Gysin maps, duality and Schubert classes}
\keywords{Push-forward, Grassmann bundle, Schubert bundle, duality theorem, Schubert classes}
\subjclass{14C17, 14M15, 14N15, 05E05}
\date{v3}
\begin{document}
\begin{abstract}
  We establish a Gysin formula for Kempf--Laksov flag bundles and we prove a duality theorem for Grassmann bundles.
  We then combine them to study Schubert bundles, their push-forwards and fundamental classes.
\end{abstract}
\maketitle
\section{Introduction}
Schubert varieties belong to the most studied algebraic varieties.
The classical Schubert calculus describes the Chow ring of the Grassmannian (parametrizing all \(d\)-planes of a fixed \(n\)-space) with the help of its Schubert classes indexed by \(\binom{n}{d}\)
partitions contained in the rectangle \((n-d)^{d}\).
Among the fundamental theorems of Schubert calculus there are: basis theorem, duality theorem and the Giambelli formula expressing a general Schubert class in terms of special ones.

Let now \(E\) be a vector bundle of rank \(n\) on a nonsingular variety \(X\) and for \(d<n\), let \(\pi\colon\G_{d}(E)\to X\) be the Grassmann bundle of \(d\)-planes in \(E\).
The basis theorem then presents the Chow ring \(A^{\bullet}(\G_{d}(E))\) of \(\G_{d}(E)\) as a free \(A^{\bullet}X\)-module of rank \(\binom{n}{d}\).
We let \(U_{d}\) (or simply \(U\) when \(d\) is fixed) denote the universal subbundle on the Grassmann bundle \(\G_{d}(E)\), as well as its pullbacks to the different varieties appearing later in the paper.
Since no confusion could arise, we will indeed drop the pullback notation (with a few exceptions).
A set of generators for \(A^{\bullet}(\G_{d}(E))\) over  \(A^{\bullet}X\) is given by the Schur classes \(s_{\beta}(U)\) of the universal subbundle for \(\beta\subseteq(n-d)^{d}\).
Intersection theory on \(\G_{d}(E)\) was studied in~\cite{Groth}, \cite{Kleiman}, \cite{Laksov}, \cite{KL}, \cite{Lascoux} and~\cite{Scott}, and more recently in~\cite{AF12}, \cite{AF15} (see also~\cite{Fulton}, \cite{FuYT} and \cite{FP}).

Given a reference flag of subbundles \(E_{1}\subsetneq\dotsb\subsetneq E_{n}=E\), where \(\rk E_{i}=i\),
for any partition \(\lambda\subseteq(n-d)^{d}\),
there is a \textsl{Schubert bundle} \(\varpi_{\lambda}\colon\Omega_\lambda(E_{\bullet})\to X\) in \(\G_{d}(E)\to X\), described by certain incidence conditions (\textit{cf.}~\eqref{eq:def_omega}).
(See Definition~\ref{defi:schubert} for details.)
It is a subvariety of \(\G_{d}(E)\). The \textsl{Schubert class} \([\Omega_{\lambda}(E_{\bullet})]\) is then the fundamental class of \(\Omega_{\lambda}(E_{\bullet})\) in \(\G_{d}(E)\).
In 1974, in a remarkable paper~\cite{KL}, Kempf and Laksov found a generalization to the relative setting of the Giambelli formula for Schubert classes.
Let \(\F(1,\ldots,d)(E)\to X\) be the flag bundle parametrizing flags of \(1\)-planes, \(2\)-planes, \ldots, \(d\)-planes in \(E\).
The Kempf--Laksov solution used certain flag bundles (which we call \textsl{Kempf--Laksov flag bundles})
\[
  \vartheta_{\mu}
  \colon
  F_{\mu}(E_{\bullet})
  \hookrightarrow
  \F(1,\dotsc,d)(E)
  \to X,
\]
indexed by strict partitions \(\mu\subseteq(n)^{d}\) with \(d\) parts.
(See Definition~\ref{defi:kempf-laksov} for details.)
The critical feature of Kempf--Laksov bundles is of course that for some strict partition \(\nu\)  depending on \(\lambda\), \textit{cf.}~\eqref{eq:def_nu}, \(\F_{\nu}(E_{\bullet})\) is a desingularization of \(\Omega_\lambda(E_\bullet)\).

In the present paper, we study various push-forward formulas related to Grassmann bundles, Schubert bundles and Kempf--Laksov bundles.
We shall work in the framework of intersection theory of~\cite{Fulton} (\textit{cf.} also~\cite[App. B]{FuYT} and~\cite[App. A]{FP}).
Recall that a proper morphism \(F\colon Y\to X\) of algebraic varieties over an algebraically closed field yields an additive map \(F_{\ast}\colon A^{\bullet}Y\to A^{\bullet}X\) of Chow groups induced by push-forward cycles (often called the \textsl{Gysin map}).
The theory developed in~\cite{Fulton} allows one to work with singular varieties, or with cohomology.
In this paper, \(X\) will be always nonsingular.
For possibly singular \(Y\), if \(P\) is a polynomial in the Chern classes of a vector bundle on \(Y\), by push-forward of \(P\) along \(F\), we shall mean \(F_{\ast}(P\cap[Y])\).
\newline

The main results of this paper are Theorem~\ref{theo:gysin} (a Gysin formula for Kempf--Laksov bundles) and Theorem~\ref{theo:duality} (a duality theorem for Grassmann bundles); other results are consequences of these two.

Sect.~\ref{se:gysin} is mainly devoted to the proof of Theorem~\ref{theo:gysin}, namely a compact formula for the Gysin map along \(\vartheta_{\mu}\).
We get an expression in the spirit of~\cite{D} and~\cite{DP1}, presenting the push-forward as some specific coefficient of a certain polynomial, depending only on the Segre classes of the reference flag of bundles \(E_{\bullet}\).

Then, in Sect.~\ref{se:duality}, we push further the study of the Gysin map along \(\pi\colon\G_{d}(E)\to X\) started in~\cite{DP1} to establish Theorem~\ref{theo:duality}, namely a formula for the push-forward along \(\pi\) of the product of two Schur polynomials \(s_{\alpha},s_{\beta}\) of the universal subbundle \(U\to\G_{d}(E)\). It extends the classical duality of (relative) Schubert calculus, since we also compute the intersection in positive degree, \textit{i.e.} when \(\abs{\alpha}+\abs{\beta}>d(n-d)\).

In Sect.~\ref{se:giambelli}, as a corollary of the Gysin formulas of Sect.~\ref{se:gysin} and Sect.~\ref{se:duality}, we give an elementary constructive method to compute Schubert classes in Grassmann bundles.
To the best of our knowledge, this method is new.
See Sect.\ref{sse:Giambelli_intro} for more detail, and further references.
\newline

Before entering the subject, a few words about the notation.
The Greek letters \(\alpha,\beta,\gamma,\delta\); \(\lambda,\mu,\nu,\rho\) always denote partitions.
For \(\ell\geq0\), we denote \((\ell)^{d}\), the rectangular partition with \(d\) parts of length \(\ell\).
We denote by \(\subseteq\) the containment relation of (the diagrams of) partitions.
 For \(d\leq n\) fixed, we let \(\rho\) denote the triangular partition \((d,d-1,\dotsc,1)\).

We will also use some operations of \(\Z^{d}\) whose results are not necessarily partitions. In particular, we write \(\alpha+\beta\) for the termwise sum and \(\alpha-\beta\) for the termwise difference of two partitions, and we denote \(\alpha^{\rev}=(\alpha_{d},\dotsc,\alpha_{1})\).

\section{Kempf--Laksov flag bundles}
\label{se:gysin}
\subsection{Desingularization of Schubert bundles}
\label{sse:desing}
Given two locally free sheaves \(A\) and \(B\), in accordance with~\cite{KL}, we shall say for short that \(A\) is a subsheaf of \(B\) and write \(A\subseteq B\) whenever \(A\) is a subsheaf that is locally a direct summand in \(B\). We may also speak of subbundles for locally free subsheaves of constant rank.

Let \(E\to X\) be a rank \(n\) vector bundle on a variety \(X\).
\begin{defi}
  For an integer \(d\leq n\), the \textsl{Grassmann bundle} \(\G_{d}(E)\) over \(X\) parametrizes subbundles \(V_{d}\subseteq E\) of rank \(d\).
  Namely \(\G_{d}(E)\) is the scheme representing the functor from \(X\)-schemes to sets
  \[
    T\mapsto\Set*{V_{d}\colon V_{d} \text{ locally free subsheaf of rank \(d\) on \(T\) such that \(V_{d}\subseteq E_{T}\)}},
  \]
  where \(E_{T}\) is the pullback of \(E\) to \(T\) (see \cite[I.9.7]{EGA} or \cite{Kleiman}).
\end{defi}
Note that over a point \(x\in X\), one has \(\G_{d}(E)(x)=\Set{V\subseteq E(x)\colon\dim(V)=d}\).
Accordingly, we will call the bundle \(\pi\colon\G_{d}(E)\to X\) the Grassmann bundle of \(d\)-planes in (the fibers of) \(E\). It comes with a universal rank \(d\) subbundle \(U\) of \(\pi^{\ast}E\).
For \(d=1\) one recovers the projective bundle of lines in \(E\), \(\P(E)\bydef\G_{1}(E)\) with the tautological line bundle \(\O_{\P(E)}(-1)\).

Assume now that \(E\) is equipped with a reference flag of subbundles \(E_{1}\subsetneq\dotsb\subsetneq E_{n}=E\), where \(\rk E_{i}=i\).
\begin{defi}
  \label{defi:schubert}
  For a partition \(\lambda\subseteq(n-d)^{d}\), the \textsl{Schubert bundle} \(\Omega_{\lambda}(E_{\bullet})\subseteq\G_{d}(E)\) over \(X\) parametrizes subbundles \(V_{d}\subseteq E\) of rank \(d\) such that \(\rk(V_{d}\cap E_{n-d-\lambda_{i}+i})\geq i\) for \(i=1,\dotsc,d\).
  Namely \(\Omega_{\lambda}(E_{\bullet})\) is the scheme representing the functor from \(X\)-schemes to sets
  \begin{multline*}
    T\mapsto\{V_{d}\colon V_{d} \text{ locally free subsheaf of rank \(d\) on \(T\) such that \(V_{d}\subseteq E_{T}\) and such that}\\
    \text{for }i=1,\dotsc,d\colon\Wedge^{n-d-\lambda_{i} +1}(E_{n-d-\lambda_{i}+i}) \to \Wedge^{n-d-\lambda_{i} +1}(E/V_{d})\text{ is zero}\},
  \end{multline*}
  where we consider the maps induced from \(E_{n-d-\lambda_{i}+i}\hookrightarrow E \twoheadrightarrow E/V_{d}\).
  It is clearly a subfunctor of the above.
\end{defi}

For any partition \(\lambda\subseteq(n-d)^{d}\), the Schubert bundle \(\varpi_{\lambda}\colon\Omega_{\lambda}(E_{\bullet})\to X\) in \(\G_{d}(E)\) is given over the point \(x\in X\) by
\begin{equation}
  \label{eq:def_omega}
  \Omega_{\lambda}(E_{\bullet})(x)
  \bydef
  \Set*{
    V\in \G_{d}(E)(x)
    \colon
    \dim(V\cap E_{n-d-\lambda_{i}+i}(x))\geq i,
    \text{ for }i=1,\dotsc,d
  }.
\end{equation}
In this description, it appears clearly that the only non-trivial conditions correspond to indices \(i\) for which \(\lambda_{i}>0\).

We denote by
\[
  (\nu_{1},\dotsc,\nu_{d})
  \bydef
  (n-d-\lambda_{1}+1,\dotsc,n-d-\lambda_{d}+d)^{\rev}
  =
  (n-d-\lambda_{d}+d,\dotsc,n-d-\lambda_{1}+1),
\]
the dimensions of the spaces of the reference flag involved in the definition of \(\Omega_{\lambda}(E_{\bullet})\)---in reverse order---.
To a partition \(\alpha\subseteq(n-d)^{d}\), one associates a dual partition \(\dual\alpha\subseteq(n-d)^{d}\) by setting
\[
  \dual\alpha
  =
  (n-d)^{d}-\alpha^{\rev},
\]
as illustrated below:
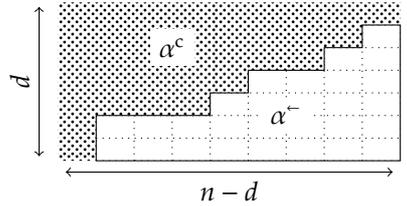
\def\Tbeta
{
  \draw(0,0)--(8,0)--(8,6)--++(-1,0)--++(0,-1)--++(-1,0)--++(0,-1)--++(-2,0)--++(0,-1)--++(-1,0)--++(0,-1)--++(-3,0)--cycle;
  \fill[pattern=crosshatch dots](-1,7)--(8,7)--(8,6)--++(-1,0)--++(0,-1)--++(-1,0)--++(0,-1)--++(-2,0)--++(0,-1)--++(-1,0)--++(0,-1)--++(-3,0)--++(0,-2)--++(-1,0)--cycle;
  \clip(0,0)--(8,0)--(8,6)--++(-1,0)--++(0,-1)--++(-1,0)--++(0,-1)--++(-2,0)--++(0,-1)--++(-1,0)--++(0,-1)--++(-3,0)--cycle;
  \draw[dotted] grid (8,6);
}
\begin{center}
  \begin{tikzpicture}[xscale=.5,yscale=.3,baseline=(current bounding box.center)]
    \draw[<->](-1.5,.2)--(-1.5,6.8)node[midway,sloped,above]{\(d\)};
    \draw[<->](-.8,-.5)--(7.8,-.5)node[midway,sloped,below]{\(n-d\)};
    \begin{scope}
      \Tbeta
    \end{scope}
    \draw (2,5) node[fill=white]{\(\dual\alpha\)};
    \draw (5,2) node[fill=white]{\(\alpha^{\rev}\)};
  \end{tikzpicture}
  .
\end{center}
With this notation
\begin{equation}
  \label{eq:def_nu}
  \nu
  \bydef
  \dual\lambda+\rho.
\end{equation}
So \(\nu\) is a strict partition with \(d\) parts, and furthermore, \(\rho_{i}\leq\nu_{i}\leq\nu_{1}=n-\lambda_{d}\leq n\), for any \(i\).

Note that the above description of \(\Omega_{\lambda}(E_{\bullet})\) can be restated using the strict partition with \(d\) parts \(\nu\subseteq(n)^{d}\) with the conditions
\begin{equation}
  \label{eq:condition_nu}
  \dim(V\cap E_{\nu_{i}}(x))\geq d+1-i,\text{ for }i=1,\dotsc,d.
\end{equation}

Schubert bundles can be singular. One natural way to resolve the singularities is to use some flag bundles (\cite{KL}).
Let us first recall the following classical construction.
\begin{defi}
  For an integer \(d\leq n\), the \textsl{flag bundle} \(\F(1,\dotsc,d)(E)\) over \(X\) parametrizes flags of subbundles \(V_{1}\subsetneq\dotsb\subsetneq V_{d}\subseteq E\) with \(\rk(V_{i})=i\).
  Namely, \(\F(1,\dotsc,d)(E)\) is the scheme representing the functor from \(X\)-schemes to sets
  \[
    T\mapsto\Set*{V_{1}\subsetneq\dotsb\subsetneq V_{d}\colon V_{i} \text{ locally free subsheaf of rank \(i\) on \(T\) such that \(V_{i}\subseteq E_{T}\)}}.
  \]
\end{defi}
Then one can introduce the central objects of this section.
\begin{defi}[\cite{KL}]
  \label{defi:kempf-laksov}
  For a strict partition \(\mu\subseteq(n)^{d}\) with \(d\) parts,
  the \textsl{Kempf--Laksov flag bundle} \(F_{\mu}(E_{\bullet})\subseteq\F(1,\dotsc,d)(E)\) over \(X\) parametrizes flags of subbundles \(V_{1}\subsetneq\dotsb\subsetneq V_{d}\) with \(\rk(V_{i})=i\) such that \(V_{i}\subseteq E_{\mu_{i}}\).
  Namely, \(F_{\mu}(E_{\bullet})\) is the scheme representing the functor from \(X\)-schemes to sets
  \[
    T\mapsto\Set*{V_{1}\subsetneq\dotsb\subsetneq V_{d}\colon V_{i} \text{ locally free subsheaf of rank \(i\) on \(T\) such that \(V_{i}\subseteq (E_{\mu_{i}})_{T}\)}}.
  \]
\end{defi}
The Kempf--Laksov flag bundle \(\vartheta_{\mu}\colon F_{\mu}(E_{\bullet})\to X\) is given over the point \(x\in X\) by
\begin{equation}
  \label{eq:def_f}
  F_{\mu}(E_{\bullet})(x)
  \bydef
  \Set*{
    0\subsetneq V_{1}\subsetneq\dotsb\subsetneq V_{d} \in \F(1,\dotsc,d)(E)(x)
    \colon
    V_{d+1-i}\subseteq E_{\mu_{i}}(x),
    \text{ for }i=1,\dotsc,d
  }.
\end{equation}

These bundles appear naturally as desingularizations of Schubert bundles (see~\cite{KL}).
For a partition \(\lambda\subseteq(n-d)^{d}\), denoting \(\nu=\dual\lambda+\rho\) as above, by \eqref{eq:condition_nu}, the forgetful map \(\F(1,\dotsc,d)(E)\to\G_{d}(E)\) induces a birational map \(\varphi\colon F_{\nu}(E_{\bullet})\to\Omega_{\lambda}(E_{\bullet})\); on the \textsl{Schubert cell} given over the point \(x\in X\) by
\[
  \mathring\Omega_{\lambda}(E_{\bullet})(x)
  \bydef
  \Set*{
    V\in \G_{d}(E)(x)
    \colon
    \dim(V\cap E_{\nu_{i}}(x))= d+1-i,\text{ for } i=1,\dotsc,d
  },
\]
which is open dense in \(\Omega_{\lambda}(E_{\bullet})\), the inverse map is
\(
V\mapsto (V\cap E_{\nu_{d}}(x),\dotsc, V\cap E_{\nu_{1}}(x))
\).

Later, to complete the study of Schubert bundles, we will fix a partition \(\lambda\) and consider \(F_{\nu}(E_{\bullet})\) for \(\nu=\dual\lambda+\rho\), but we shall first study Kempf--Laksov bundles \(F_{\mu}(E_{\bullet})\to X\) in themselves.

\subsection{Gysin formulas for Kempf--Laksov flag bundles}
Regarding our goals, a central feature of Kempf--Laksov flag bundles is that these can be regarded as chains of projective bundles of lines defined by the reference flag of bundles \(E_{\bullet}\) and the universal subbundles \(U_{1},\dotsc,U_{d}\).
\begin{lemm}
  \label{lemm:chain}
  Let \(\mu\subseteq(n)^{d}\) be a strict partition with \(d\) parts.
  For \(e=1,\dotsc,d-1\),
  the forgetful map \(\F(1,\dotsc,d-e+1)(E)\to\F(1,\dotsc,d-e)(E)\)
  induces a map
  \(F_{(\mu_{e},\dotsc,\mu_{d})}(E_{\bullet})\to F_{(\mu_{e+1},\dotsc,\mu_{d})}(E_{\bullet})\),
  isomorphic to \(\P(E_{\mu_{e}}/U_{d-e})\).
  Lastly, the bundle \(F_{(\mu_{d})}(E_{\bullet})\to X\) is isomorphic to \(\P(E_{\mu_{d}})\).
\end{lemm}
\begin{proof}
  Having defined the flag bundles globally, it is sufficient to check the assertions locally.
  Fix \(x\in X\).
  Over a point \((V_{1},\dotsc,V_{d-e})\in F_{(\mu_{e+1},\dotsc,\mu_{d})}(E_{\bullet})(x)\), one has
  \(V_{d-e}\subseteq E_{\mu_{e+1}}(x)\subseteq E_{\mu_{e}}(x)\) and
  the fiber of \(F_{(\mu_{e},\dotsc,\mu_{d})}(E_{\bullet})\to F_{(\mu_{e+1},\dotsc,\mu_{d})}(E_{\bullet})\) over
  \((V_{1},\dotsc,V_{d-e})\) consists of subspaces \(V_{d-e+1}\) such that
  \(V_{d-e}\subsetneq V_{d-e+1}\subseteq E_{\mu_{e}}(x)\).
  Since \(\dim V_{i}=i\), the result follows.
\end{proof}

To sum up, we obtain a chain of projective bundles of lines
\[
  F_{(\mu_{1},\dotsc,\mu_{d})}(E_{\bullet})
  \to
  F_{(\mu_{2},\dotsc,\mu_{d})}(E_{\bullet})
  \to
  \dotsb
  \to
  F_{(\mu_{d-1},\mu_{d})}(E_{\bullet})
  \to
  F_{(\mu_{d})}(E_{\bullet})
  \to
  X,
\]
which is the same as
\begin{equation}
  \label{eq:chain}
  \P(E_{\mu_{1}}/U_{d-1})
  \to
  \P(E_{\mu_{2}}/U_{d-2})
  \to
  \dotsb
  \to
  \P(E_{\mu_{d-1}}/U_{1})
  \to
  \P(E_{\mu_{d}})
  \to
  X.
\end{equation}
As in~\cite{DP1}, one can deduce a Gysin formula for \(F_{\mu}(E_{\bullet})\to X\) from the described structure of chain of projective bundles \eqref{eq:chain}.
For the sake of completeness, let us recall the main lines of the argument.

For a Laurent polynomial \(P\) in \(d\) variables \(t_{1},\dotsc,t_{d}\), and a monomial \(m\), we denote by \([m](P)\) the coefficient of \(m\) in the expansion of \(P\). Clearly, for any second monomial \(m'\), one has \([mm'](Pm')=[m](P)\), a property that we will use repeatedly.
For a vector bundle \(E\to X\) of rank \(r\),  recall the Gysin formula for the projective bundle of lines \(p\colon\P(E)\to X\)
\begin{equation}
  \label{eq:P}
  p_{\ast}(\xi^{i})
  =
  s_{i-r+1}(E)
  =
  [t^{r-1}](t^{i}s_{1/t}(E)),
\end{equation}
where \(\xi\bydef c_{1}(\O_{\P(E)}(1))\), \(s_{i}(E)\) is the \(i\)th Segre class, and \(s_{1/t}(E)\) the Segre polynomial evaluated in \(1/t\) (this yields a Laurent polynomial). The idea is simply to iterate this formula.

In our push-forward formulas, for \(d\) fixed and for a symmetric polynomial \(f\) in \(d\) variables with coefficients in \(A^{\bullet}X\), by \(f(U)\) we shall mean \(f\) specialized with the Chern roots of \(U^{\vee}\) (because the fundamental formula~\eqref{eq:P} use \(\xi=c_{1}(\O_{\P(E)}(-1)^{\vee})\)).

\begin{theo}
  \label{theo:gysin}
  For a rank \(n\) vector bundle \(E\to X\) on a variety \(X\) with a reference flag \(E_{\bullet}\) on it, and for a strict partition \(\mu\subseteq(n)^{d}\) with \(d\leq n\) parts,
  the push-forward along \(\vartheta_{\mu}\colon F_{\mu}(E_{\bullet})\to X\) of a symmetric polynomial \(f\) in \(d\) variables with coefficients in \(A^{\bullet}X\) is
  \[
    (\vartheta_{\mu})_{\ast}(f(U))
    =
    \Big[\tprod_{i=1}^{d} t_{i}^{\mu_{i}-1}\Big]
    \left(
    f(t_{1},\dotsc,t_{d})
    \tprod_{1\leq i<j\leq d}(t_{i}-t_{j})
    \tprod_{1\leq i\leq d}s_{1/t_{i}}(E_{\mu_{i}})
    \right).
  \]
\end{theo}
\begin{proof}
  We enumerate the Chern roots \(\xi_{1},\dotsc,\xi_{d}\) of \(U^{\vee}\) by taking
  \[
    \xi_{i}
    \bydef
    -c_{1}(U_{d+1-i}/U_{d-i}).
  \]
  For notational convenience, we will denote \(\underline{t\xi}_{e}\bydef(t_{1},\dotsc,t_{e},\xi_{e+1},\dotsc,\xi_{d})\) the \(d\)-tuple obtained from \((\xi_{1},\dotsc,\xi_{d})\) after replacement of the first \(e\) roots \(\xi_{i}\) by formal variables \(t_{i}\).

  For \(e_{1}\leq e_{2}\leq d-1\), let us denote
  by \(\int_{e_{1}}^{e_{2}}\) the Gysin map along \(F_{(\mu_{e_{1}+1},\dotsc,\mu_{d})}(E_{\bullet})\to F_{(\mu_{e_{2}+1},\dotsc,\mu_{d})}(E_{\bullet})\)
  and for \(e_{2}=d\), let us denote
  by \(\int_{e_{1}}^{d}\) the Gysin map along \(F_{(\mu_{e_{1}+1},\dotsc,\mu_{d})}(E_{\bullet})\to X\).

  We will prove by induction on \(e=0,\dotsc,d\) that
  \[
    \tag{\ensuremath{\ast}}
    \int_{0}^{e}\!\!f(U)
    =
    \big[\tprod_{i=1}^{e}t_{i}^{\mu_{i}-(d+1-e)}\big]
    \Big(
    f\big(\underline{t\xi}_{e}\big)
    \tprod_{1\leq i<j\leq e}(t_{i}-t_{j})
    \tprod_{1\leq i\leq e}s_{1/t_{i}}(E_{\mu_{i}}-U_{d-e})
    \Big).
  \]
  which for \(e=d\) is the announced result,
  since \(\int_{0}^{d}=(\vartheta_{\mu})_{\ast}\).

  For \(e=0\), this is the definition of \(f(U)\).
  Assume that the formula holds for \(e<d\).

  By Lemma \ref{lemm:chain},
  \(F_{(\mu_{e+1},\dotsc,\mu_{d})}(E_{\bullet})\to F_{(\mu_{e+2},\dotsc,\mu_{d})}(E_{\bullet})\) is isomorphic to \(\P(E_{\mu_{e+1}}/U_{d-(e+1)})\),
  with the notation of \eqref{eq:P}, this projective bundle has rank
  \[
    r-1
    =
    \mu_{e+1}-(d-(e+1))-1,
  \]
  so by \eqref{eq:P} one has
  \[
    \tag{\(\ast_{2}\)}
    \int_{e}^{e+1}
    \xi_{e+1}^{i}
    =
    [t_{e+1}^{\mu_{e+1}-(d+1-(e+1))}]
    \Big(
    t_{e+1}^{i}
    s_{1/t_{e+1}}(E_{\mu_{e+1}}-U_{d-(e+1)})
    \Big).
  \]
  Now by the induction hypothesis (\(\ast\)),
  \[
    \int_{0}^{e+1}f(U)
    =
    \int_{e}^{e+1}
    \int_{0}^{e}f(U)
    =
    \big[\tprod_{i=1}^{e}t_{i}^{\mu_{i}-(d+1-e)}\big]
    \Big(
    \int_{e}^{e+1}
    P(\underline{t\xi}_{e})
    \Big),
  \]
  where \(P=P(E_{\mu_{1}},\dotsc,E_{\mu_{e}})\) is the Laurent polynomial in \(d\) variables such that
  \[
    P(\underline{t\xi}_{e})
    =
    f(\underline{t\xi}_{e})
    \tprod_{1\leq i<j\leq e}(t_{i}-t_{j})
    \tprod_{1\leq i\leq e}s_{1/t_{i}}(E_{\mu_{i}}-U_{d-e}).
  \]
  Note that the Segre classes of the universal bundle \(U_{d-e}\) are polynomials in \(\xi_{e+1},\dotsc,\xi_{d}\).
  To apply (\(\ast_{2}\)), we regard \(P(\underline{t\xi}_{e})\) as a polynomial in \(\xi_{e+1}\) with coefficients in \(A^{\bullet}(F_{(\mu_{e+2},\dotsc,\mu_{d})}(E_{\bullet}))\), according to the formula
  \[
    P(\underline{t\xi}_{e})
    =
    f\big(\underline{t\xi}_{e}\big)
    \tprod_{1\leq i<j\leq e}(t_{i}-t_{j})
    \tprod_{1\leq i\leq e}s_{1/t_{i}}(E_{\mu_{i}}-U_{d-(e+1)})(1-{\xi_{e+1}}/{t_{i}}).
  \]
  Hence, using (\(\ast_{2}\))
  \begin{align*}
    \int_{e}^{e+1}
    P(\underline{t\xi}_{e})
    &=
    [t_{e+1}^{\mu_{e+1}-(d+1-(e+1))}]
    \Big(
    f\big(\underline{t\xi}_{e+1}\big)
    \tprod_{1\leq i<j\leq e}(t_{i}-t_{j})
    \tprod_{1\leq i\leq e+1}s_{1/t_{i}}(E_{\mu_{i}}-U_{d-(e+1)})
    \tprod_{1\leq i\leq e}(1-{t_{e+1}}/{t_{i}})
    \Big)
    \\&=
    [t_{e+1}^{\mu_{e+1}-(d+1-(e+1))}]
    \Big(
    \frac{1}{t_{1}\dotsm t_{e}}
    f\big(\underline{t\xi}_{e+1}\big)
    \tprod_{1\leq i<j\leq e+1}(t_{i}-t_{j})
    \tprod_{1\leq i\leq e+1}s_{1/t_{i}}(E_{\mu_{i}}-U_{d-(e+1)})
    \Big).
  \end{align*}

  It follows that
  \[
    \int_{0}^{e+1}f(U)
    =
    \big[\tprod_{i=1}^{e}t_{i}^{\mu_{i}-(d+1-e)}
    t_{e+1}^{\mu_{e+1}-(d+1-(e+1))}\big]
    \Big(
    \frac{1}{t_{1}\dotsm t_{e}}
    f\big(\underline{t\xi}_{e+1}\big)
    \tprod_{1\leq i<j\leq e+1}(t_{i}-t_{j})
    \tprod_{1\leq i\leq e+1}s_{1/t_{i}}(E_{\mu_{i}}-U_{d-(e+1)})
    \Big).
  \]
  Multiplying the extracted monomial and the polynomial by \(t_{1}\dotsm t_{e}\) one obtains
  \[
    \int_{0}^{e+1}f(U)
    =
    \big[\tprod_{i=1}^{e}t_{i}^{\mu_{i}-(d+1-(e+1))}
    t_{e+1}^{\mu_{e+1}-(d+1-(e+1))}\big]
    \Big(
    f\big(\underline{t\xi}_{e+1}\big)
    \tprod_{1\leq i<j\leq e+1}(t_{i}-t_{j})
    \tprod_{1\leq i\leq e+1}s_{1/t_{i}}(E_{\mu_{i}}-U_{d-(e+1)})
    \Big).
  \]
  This is (\(\ast\)) for \(e+1\), whence by induction the formula (\(\ast\)) holds for any \(e=0,\dotsc,d\), and this finishes the proof.
\end{proof}

Note that the assertion holds for any polynomial \(f\) in \(\xi_{1},\dotsc,\xi_{d}\) with coefficients in \(A^{\bullet}X\), with the same proof, but in applications to Schubert calculus in Grassmann bundles, one shall need only symmetric polynomials.

\subsection{Push-forward of Schur classes}
\label{sse:gysin_schur}
We can now specialize Theorem~\ref{theo:gysin}, to get a formula for the push-forward of Schur polynomials. This is done in Proposition~\ref{prop:gysin}.

For any partition \(\alpha=(\alpha_{1},\dotsc,\alpha_{d})\), recall that the \textsl{Schur polynomial} \(s_{\alpha}\in\Z[t_{1},\dotsc,t_{d}]\) can be defined by the formula
\[
  s_{\alpha}(t_{1},\dotsc,t_{d})
  \bydef
  \frac{\det\big(t_{j}^{\alpha_{i}+d-i}\big)_{1\leq i,j\leq d}}
  {{\displaystyle\tprod_{1\leq i<j\leq d}}(t_{i}-t_{j})}.
\]
Note that in particular for \(\alpha=(i)\), one has \(s_{i}\) the complete symmetric function of degree \(i\).
It is convenient to set \(s_{i}=0\) for \(i<0\).
Then the Jacobi--Trudi identity states
\[
  s_{\alpha}(t_{1},\dotsc,t_{d})
  =
  \det
  \Big(
  s_{\alpha_{i}-i+j}(t_{1},\dotsc,t_{d})
  \Big)_{1\leq i,j\leq d}.
\]

One can generalize further the Segre classes by considering skew Schur functions, defined for two partitions \(\alpha\) and \(\beta\) by
\begin{equation}
  \label{eq:skew}
  s_{\skew{\alpha}{\beta}}(E)
  \bydef
  \det
  \Big(
  s_{\alpha_{i}-i+j-\beta_{j}}(E)
  \Big)_{1\leq i,j\leq d}.
\end{equation}

\begin{prop}
  \label{prop:gysin}
  For a strict partition \(\mu\subseteq(n)^{d}\) with \(d\) parts, and any partition \(\alpha\),
  the push-forward along \(\vartheta_{\mu}\colon F_{\mu}(E_{\bullet})\to X\) of the Schur class \(s_{\alpha}(U)\) is
  \[
    (\vartheta_{\mu})_{\ast}(s_{\alpha}(U))
    =
    \det\big(
    s_{\alpha_{i}-i+d+1-\mu_{j}}(E_{\mu_{j}})
    \big)_{1\leq i,j\leq d}.
  \]
\end{prop}
\begin{proof}
  Applying the result of Theorem~\ref{theo:gysin} to \(f=s_{\alpha}\) one gets
  \[
    (\vartheta_{\mu})_{\ast}(s_{\alpha}(U))
    =
    [t_{d}^{\mu_{d}-1}\dotsb t_{1}^{\mu_{1}-1}]
    \left(
    s_{\alpha}(t_{1},\dotsc,t_{d})
    \tprod_{1\leq i<j\leq d}(t_{i}-t_{j})
    \tprod_{1\leq i\leq d}s_{1/t_{i}}(E_{\mu_{i}})
    \right).
  \]
  Now, by definition
  \[
    s_{\alpha}(t_{1},\dotsc,t_{d})\prod_{1\leq i<j\leq d}(t_{i}-t_{j})
    =
    \det\big(t_{j}^{\alpha_{i}+d-i}\big)_{1\leq i,j\leq d}.
  \]
  Hence, dividing the \(j\)th column of the determinant and also the extracted monomial by \(t_{j}^{\mu_{j}-1}\), for \(j=1,\dotsc,d\),  one gets
  \[
    (\vartheta_{\mu})_{\ast}(s_{\alpha}(U))
    =
    [1]
    \left(
    \tprod_{1\leq j\leq d}s_{1/t_{j}}(E_{\mu_{j}})
    \det\big(t_{j}^{\alpha_{i}+d-i+1-\mu_{j}}\big)_{1\leq i,j\leq d}
    \right).
  \]
  Then using again the linearity of the determinant with respect to columns as in~\cite[Lemma~4.1]{DP1}, one obtains:
  \[
    (\vartheta_{\mu})_{\ast}(s_{\alpha}(U))
    =
    \det\big(
    [1]\big(
    t_{j}^{\alpha_{i}+d-i+1-\mu_{j}}
    s_{1/t_{j}}(E_{\mu_{j}})
    \big)
    \big)_{1\leq i,j\leq d}
    =
    \det\big(
    s_{\alpha_{i}-i+d+1-\mu_{j}}(E_{\mu_{j}})
    \big)_{1\leq i,j\leq d}.
  \]
  This is the announced formula.
\end{proof}

\section{Duality for Grassmann bundles}
\label{se:duality}
We investigate the combinatorics of partitions with at most \(d\) parts, and deduce a strong version of the duality theorem in Grassmann bundles.
We use Young tableaux. For all terminology and standard results used in this section, we refer the reader to~\cite{FuYT}.

\subsection{Littlewood--Richardson numbers}
By definition, the Littlewood--Richardson number \(c_{\alpha\,\beta}^{\gamma}\) associated to three partitions \(\alpha\), \(\beta\), \(\gamma\) is the coefficient \(\langle s_{\alpha}s_{\beta},s_{\gamma}\rangle\) of \(s_{\gamma}\) in the decomposition of the symmetric function \(s_{\alpha}s_{\beta}\) over the Schur basis. It is thus symmetric in \(\alpha\) and \(\beta\). The coefficient \(c_{\alpha\,\beta}^{\gamma}\) is zero if \(\alpha\not\subseteq\gamma\) or \(\beta\not\subseteq\gamma\).

If \(\beta\subseteq\alpha\), the number \(c_{\beta\,\gamma}^{\alpha}\) is also the coefficient \(\langle s_{\skew{\alpha}{\beta}},s_{\gamma}\rangle\) of \(s_{\gamma}\) in the decomposition of the symmetric function \(s_{\skew{\alpha}{\beta}}\) over the Schur basis.
It follows that \(c_{\beta\,\gamma}^{\alpha}=c_{\gamma\,\beta}^{\alpha}\) depends only on the partition \(\beta\) and the skew partition \(\skew{\alpha}{\gamma}\).
This observation implies the following.
\begin{lemm}
  \label{lemm:L-R}
  Let \(\alpha, \beta\) be partitions with at most \(d\) parts.
  For any partition \(\gamma\) and any rectangular partition with \(d\) parts \(\square\subseteq\gamma\), one has:
  \[
    c_{\beta\,\gamma}^{\alpha}
    =
    c_{\beta\,(\gamma-\square)}^{\alpha-\square}.
  \]
\end{lemm}
\begin{proof}
  Assume \(\square\subseteq\gamma\subseteq\alpha\) (otherwise both coefficients in the sought formula are zero).
  The skew shapes \(\skew{\alpha}{\gamma}\) and \(\skew{{(\alpha-\square)}}{{(\gamma-\square)}}\) are then the same, and this finishes the proof.
\end{proof}

\subsection{A product formula}
We shall now give a modern treatment of a theorem of Jacobi (1840) and Naegelbasch (1871) asserting that the product of two Schur functions in \(d\) variables can be expressed as a determinant of order \(d\) in \(s_{i}\) (see~\cite[p.~188]{Lascoux77} and see \cite[I.3.8]{Macdonald} for another proof).
Using \eqref{eq:skew}, one can rather interpret this determinant as a skew Schur function.
\begin{lemm}
  \label{lemm:prod/skew}
  Let \(\alpha\) and \(\beta\) be partitions with at most \(d\) parts.
  For any rectangular partition with \(d\) parts \(\square\supseteq \beta\), one has
  \[
    s_{\alpha}(t_{1},\dotsc,t_{d})s_{\beta}(t_{1},\dotsc,t_{d})
    =
    s_{\skew{\square+\alpha}{\square-\beta^{\rev}}}(t_{1},\dotsc,t_{d}).
  \]
\end{lemm}
\begin{proof}
  For any skew Schur diagram \(\delta\)
  \[
    s_{\delta}(t_{1},\dotsc,t_{d})
    =
    \sum_{[T]=\delta}t^{T},
  \]
  where the sum is over the semi-standard tableaux \(T\) with shape \(\delta\) and values in \(\Set{1,\dotsc,d}\) and where for such a tableau \(t^{T}\) denotes the monomial \(t_{1}^{m_{1}}\dotsm t_{d}^{m_{d}}\), with \(m_{i}\) the number of entries of \(T\) with value \(i\) for \(i=1,\dotsc,d\).

  As a consequence, it is sufficient to show that there is a bijection between the pairs of tableaux with respective shapes \(\alpha,\beta\) and the tableaux with shape \(\skew{\square+\alpha}{\square-\beta^{\rev}}\).

  Given a filling \(T_{\alpha}\) of \(\alpha\) together with a filling \(T_{\beta}\) of \(\beta\),
  we apply the following process to \(T_{\beta}\):
  replace each entry \(v\) by \(d+1-v\) and rotate the diagram by half a turn.
  This yields a skew tableau \(T_{\beta}'\) with shape \([T_{\beta}']=\skew{\square}{\square-\beta^{\rev}}\).
  Then the concatenation of \(T_{\beta}'\) and \(T_{\alpha}\) is a tableau with shape \(\skew{\square+\alpha}{\square-\beta^{\rev}}\) (see Figure~\ref{fig:skew}).
  \begin{figure}[!ht]
    \centering
    \def\Talpha
    {
      \draw (0,1)--(0,7)--(9,7)--++(0,-3)--++(-4,0)--++(0,-2)--++(-1,0)--++(0,-1)--++(-2,0)--cycle (0,0)--(0,1);
      \clip (0,1)--(0,7)--(9,7)--++(0,-3)--++(-4,0)--++(0,-2)--++(-1,0)--++(0,-1)--++(-2,0)--cycle;
      \draw[dotted] grid (9,7);
    }
    \resizebox{\linewidth}{!}{
      \begin{tikzpicture}[xscale=.8,yscale=.5]
	\path (-1,16)--(7,16) node[midway,above]{\(T_{\beta}\)};
	\path (9,16)--(18,16) node[midway,above]{\(T_{\alpha}\)};
	\path (8,7) node[above]{\(T_{\skew{\square+\alpha}{\square-\beta^{\rev}}}\)};
	\draw[<->](-1.5,.2)--(-1.5,6.8)node[midway,sloped,above]{\(d\)};
	\begin{scope}[xshift=9cm,yshift=9cm] \Talpha \end{scope}
	\begin{scope}[xshift=8cm] \Talpha \end{scope}
	\begin{scope}[rotate=180,xshift=-7cm, yshift=-16cm] \Tbeta \end{scope}
	\begin{scope} \Tbeta \end{scope}
	\path(3,5)node[fill=white]{\(\square-\beta^{\rev}\)};
	\foreach\i in {1,...,6}
	{
	  \path(9,16-\i+1)--(9,16-\i)node[midway,right]{\(\geq\i\)};
	  \path(8,7-\i+1)--(8,7-\i)node[midway,right]{\(\geq\i\)};
	}
	\foreach\i in {1,...,6}
	{
	  \path(-1,16-\i+1)--(-1,16-\i)node[midway,right]{\(\geq\i\)};
	}
	\foreach\i in {2,...,7}
	{
	  \path(7,7-\i+1)--(7,7-\i)node[midway,right]{\(\leq\i\)};
	}
	\draw[<->,very thick](4.5,3.5)to[bend right=80]node[fill=white,below left=1.2,midway]{\(v\leftrightarrow d+1-v\)}++(-2,9);
	\draw[<->,very thick](10.5,4.5)to[bend right=80]node[fill=white,left=1.2,midway]{same}++(1,9);
      \end{tikzpicture}
    }
    \caption{The sought bijection}
    \label{fig:skew}
  \end{figure}

  Indeed, in the last column of \(T_{\beta}'\) the entry of the \(i\)th row has value \(\leq (d+1)-(d+1-i)=i\) and in the first column of \(T_{\alpha}\) the entry of the \(i\)th row has value \(\geq i\).

  The inverse map is the obvious one. Take the filling of the intersection between \(\skew{\square+\alpha}{\square-\beta^{\rev}}\) and \(\square\), rotate it, and apply the involution \(v\leftrightarrow d+1-v\) to get \(T_{\beta}\), then take the filling of the remaining part in \(\skew{\square+\alpha}{\square-\beta^{\rev}}\) to get \(T_{\alpha}\).
\end{proof}

We can deduce a rule for Littlewood--Richardson numbers.
\begin{coro}
  \label{coro:prod/skew}
  With the same notation and hypotheses, for any partition \(\gamma\) with at most \(d\) parts, the following Littlewood--Richardson numbers coincide:
  \[
    c_{\alpha\,\beta}^{\gamma}
    =
    c_{(\square-\beta^{\rev})\,\gamma}^{\square+\alpha}.
  \]
\end{coro}
\begin{proof}
  Indeed, one has
  \(
  \langle s_{\alpha}s_{\beta},s_{\gamma}\rangle
  =
  \langle s_{\skew{\square+\alpha}{\square-\beta^{\rev}}},s_{\gamma}\rangle
  \).
\end{proof}

\subsection{A strong version of the duality theorem}
The combination of Lemma~\ref{lemm:L-R} and Corollary \ref{coro:prod/skew} yields the following proposition.
\begin{prop}
  \label{prop:LR}
  For \(\alpha,\beta,\gamma\) partitions with at most \(d\) parts, if \(\square_{\beta}\supseteq\beta\) and \(\square^{\gamma}\subseteq\gamma\) are two rectangular partitions with \(d\) parts, then
  \[
    c_{\alpha\,\beta}^{\gamma}
    =
    c_{(\square_{\beta}-\beta^{\rev})\,(\gamma-\square^{\gamma})}^{\alpha+\square_{\beta}-\square^{\gamma}}.
  \]
\end{prop}

We derive the following statement.
\begin{theo}[Strong duality theorem]
  \label{theo:duality}
  Let \(E\to X\) be a rank \(n\) vector bundle over a variety, and let \(\alpha,\beta\) be two partitions with at most \(d\) parts.
  For any rectangle partition \((\ell)^{d}\supseteq\beta\),
  one has the following pushforward formula along
  \(\pi\colon\G_{d}(E)\to X\):
  \[
    \pi_{\ast}(s_{\alpha}(U)s_{\beta}(U))
    =
    s_{\skew{(\ell-n+d)^{d}+\alpha}{(\ell)^{d}-\beta^{\rev}}}(E).
  \]
  In particular, if \(\beta\subseteq(n-d)^{d}\):
  \[
    \pi_{\ast}(s_{\alpha}(U)s_{\beta}(U))
    =
    s_{\skew{\alpha}{\dual\beta}}(E).
  \]
\end{theo}
Note that using \eqref{eq:skew}, the determinant \(s_{\skew{\alpha}{\dual\beta}}(E)\) is easily shown to be indeed symmetric in \(\alpha\) and \(\beta\).
One direct consequence of our theorem is the following.
\begin{coro}[Duality theorem]
  For a partition \(\alpha\) and a partition \(\beta\subseteq(n-d)^{d}\) one has
  \[
    \pi_{\ast}(s_{\alpha}(U)s_{\beta}(U))
    =
    \begin{cases}
      1&\text{if }\alpha=\dual\beta\\
      0&\text{if }\alpha\not\supseteq\dual\beta.
    \end{cases}
  \]
\end{coro}
Under the further assumption \(\abs{\alpha}+\abs{\beta}\leq d(n-d)\), this is the standard duality theorem in Schubert calculus (see~\cite[Prop.~14.6.3]{Fulton}).
\begin{proof}[Proof of the theorem]
  One has
  \[
    s_{\alpha}(U)s_{\beta}(U)
    =
    \sum c_{\alpha\,\beta}^{\gamma}s_{\gamma}(U),
  \]
  and since \(\rk(U)=d\), one can restrict to the partitions \(\gamma\) with at most \(d\) parts.
  Using~\cite[Sect.~4]{DP1} to compute the push-forward, one obtains
  \[
    \pi_{\ast}(s_{\alpha}(U)s_{\beta}(U))
    =
    \sum c_{\alpha\,\beta}^{\gamma}s_{\gamma-(n-d)^{d}}(E).
  \]
  But now, one can assume that \(\gamma\supseteq(n-d)^{d}\), since otherwise the summand indexed by \(\gamma\) does not contribute, whence after a change of variable
  \[
    \pi_{\ast}(s_{\alpha}(U)s_{\beta}(U))
    =
    \sum c_{\alpha\,\beta}^{(\gamma+(n-d)^{d})}s_{\gamma}(E).
  \]
  Applying Proposition \ref{prop:LR} with \(\square_{\beta}=(\ell)^{d}\) and \(\square^{\gamma}=(n-d)^{d}\), one obtains
  \[
    \pi_{\ast}(s_{\alpha}(U)s_{\beta}(U))
    =
    \sum c_{((\ell)^{d}-\beta^{\rev})\,\gamma}^{\alpha+(\ell-n+d)^{d}}s_{\gamma}(E)
    =
    s_{\skew{\alpha+(\ell-n+d)^{d}}{(\ell)^{d}-\beta^{\rev}}}(E),
  \]
  and this finishes the proof.
\end{proof}

Note that a related formula was established in~\cite{JLP}: the push-forward of the product of a Schur polynomial of the universal subbundle by another Schur polynomial of the universal quotient bundle.
The formula of~\cite{JLP} is easily shown to be a consequence of our new formula.
We denote \(\lambda^{\sim}\) the conjugate partition of a partition \(\lambda\). We denote by \(\lambda\sqcup\gamma\) the concatenation of partitions \(\lambda\) and \(\gamma\); it is in general not a partition.
Let us denote by \(Q\) the universal quotient bundle on \(\G_{d}(E)\stackrel{\pi}\to X\).
For a partition \(\alpha\) with at most \(d\) parts and a partition \(\beta\) with at most \((n-d)\) parts, one has
\[
  \pi_{\ast}(s_{\alpha}(U)s_{\beta}(Q))
  =
  \sum_{\mu\subseteq(n-d)^{d}}
  (-1)^{\abs{\mu}}
  s_{\skew{\alpha}{\dual\mu}}(E)s_{\skew{\beta}{\mu^{\sim}}}(E)
  =
  s_{(\alpha-(n-d)^{d})\sqcup\beta}(E).
\]
The first equality follows from the relation \(Q=E-U\) in the Grothendieck group of \(X\), applying our push-forward formula to each factor of the expansion of \(s_{\alpha}(U)s_{\beta}(E-U)\).
The second equality is shown using Laplace expansion for Schur functions in terms of skew Schur functions.
Remark that by permuting the rows, the right hand side of the formula can be written
\[
  (-1)^{d(n-d)}
  s_{(\beta-(d)^{n-d})\sqcup\alpha}(E),
\]
which is closer to the expression of~\cite{JLP} (our convention of signs is different).

\section{Gysin maps and Schubert classes}
\label{se:giambelli}
A classical problem in Schubert calculus is to give a formula for Schubert classes.
Such formulas were given by Kempf--Laksov~\cite{KL}, Lascoux~\cite{Lascoux} in the framework of classical intersection theory, and Anderson--Fulton~\cite{AF,Anderson} in the framework of equivariant cohomology in intersection theory.
In~\cite{LT09}, Laksov and Thorup rephrase the Kempf--Laksov formula in terms of the exterior algebra of a suitable module.
In~\cite{AF15}, Anderson and Fulton gave also a variety of Kempf--Laksov formulas for degeneracy loci.

Note that like the approaches of~\cite{KL},~\cite{Lascoux},~\cite{AF15}, and others, our approach uses the desingularization of Schubert varieties (or degeneracy loci) by chains of projective bundles, and various Gysin formulas for them.
The flag bundles \(F_{\nu}(E_{\bullet})\) are well-known desingularizations of Schubert bundles \(\Omega_{\lambda}(E_{\bullet})\), which were used by many authors.
We shall also use them, but in a different way.
The most obvious difference is that the reasoning, directly over the base variety \(X\), is based on the plain comparison of two Gysin formulas for Schubert bundles, derived from our two main new results.
The formula for Schubert classes appear then as an elementary corollary, the construcitve proof of which relies on solving invertible triangular systems of equations.

\subsection{An elementary approach to compute Schubert classes}
\label{sse:Giambelli_intro}
This approach fits in a broader context than this of the present paper, and has the noteworthy feature that, like in~\cite{AF15}, it should be adaptable to the symplectic and orthogonal settings, with some technical adjustments.
The main idea is summarized in the following Figure~\ref{fig:sumup}.
\begin{figure}[!ht]
  \centering
  \subfloat[\label{fig:old}]{
    \begin{tikzpicture}[baseline=0pt]
      \node (Fl) at (2.5,2) {\(\F(1,\dotsc,d)(E)\)};
      \node (F) at (0,2) {\(F_{\nu}(E_{\bullet})\)};
      \node (Om) at (0,.5) {\(\Omega_{\lambda}(E_{\bullet})\)};
      \node (G) at (2.5,.5) {\(\G_{d}(E)\)};
      \draw[->] (F)--(Om) node[midway,left,scale=.9]{\(\varphi\)} node[midway,right]{\tiny bir.};
      \draw[right hook-latex] (Om)--(G) node[midway,below]{\tiny subvar.};
      \draw[right hook-latex] (F)--(Fl) node[midway,below]{\tiny subvar.};
      \draw[->] (Fl)to (G);
    \end{tikzpicture}
  }
  \hfil
  \subfloat[\label{fig:new}]{
    \begin{tikzpicture}[baseline=0pt]
      \node (F) at (-2,2) {\(F_{\nu}(E_{\bullet})\)};
      \node (Om) at (0,2) {\(\Omega_{\lambda}(E_{\bullet})\)};
      \node (G) at (2,2) {\(\G_{d}(E)\)};
      \node (X) at (0,.5) {\(X\)};
      \draw[->] (F)--(Om) node[midway,above,scale=.9]{\(\varphi\)} node[midway,below]{\tiny bir.};
      \draw[right hook-latex] (Om)--(G);
      \draw[->] (G)--(X) node[midway,below right,scale=.9]{\(\pi\)};
      \draw[->] (Om)--(X) node[midway,right,scale=.9]{\(\varpi_{\lambda}\)};
      \draw[->] (F)--(X) node[midway,below left,scale=.9]{\(\vartheta_{\nu}\)};
    \end{tikzpicture}
  }
  \caption{The desingularization \protect\subref{fig:old} and how we use it \protect\subref{fig:new}.}
  \label{fig:sumup}
\end{figure}

We also recall the standard approach of Kempf and Laksov (with our notation) for comparison:
\begin{center}
  \begin{tikzpicture}[baseline=0pt]
    \node (Fl) at (3,2) {\(\G_{d}(E)\times_{X} F_{\nu}(E_{\bullet})\)};
    \node (F) at (0,2) {\(F_{\nu}(E_{\bullet})\)};
    \node (Om) at (0,.5) {\(\Omega_{\lambda}(E_{\bullet})\)};
    \node (G) at (3,.5) {\(\G_{d}(E)\)};
    \draw[->] (F)--(Om) node[midway,left,scale=.9]{\(\varphi\)} node[midway,right]{\tiny bir.};
    \draw[right hook-latex] (Om)--(G);
    \draw[right hook-latex,transform canvas={yshift=2}] (F)--(Fl) node[midway,above]{\(s\)};
    \draw[->,transform canvas={yshift=-2}] (Fl)--(F) node[midway,below]{\(p_{2}\)};
    \draw[->] (Fl)--(G) node [midway,right]{\(p_{1}\)};
  \end{tikzpicture}
\end{center}
In their approach, one works with \(\G_{d}(E)\) as the base variety. To compute the fundamental class \([\Omega_{\lambda}(E_{\bullet})]\) in the Chow group of \(\G_{d}(E)\) one first computes the fundamental class \([s(F_{\nu}(E_{\bullet}))]\) in the Chow group of \(\G_{d}(E)\times_{X}F_{\nu}(E_{\bullet})\) and then one computes the push-forward along \(p_{1}\) of this specific class.

In our approach, we work above the base variety \(X\) and we use Gysin formulas for different morphisms satisfied for all classes; we do not need to compute the class \([F_{\nu}(E_{\bullet})]\) in \(A^{\bullet}(\F(1,\dotsc,d)(E))\).
There are two ways to push a class from \(\Omega_{\lambda}(E_{\bullet})\) to \(X\). Either we use the desingularization \(F_{\nu}(E_{\bullet})\) studied in Sect.~\ref{se:gysin} and we push-forward along \(\vartheta_{\nu}\), or we regard the Schubert bundle \(\Omega_{\lambda}(E_{\bullet})\) as a subvariety of the Grassmann bundle \(\G_{d}(E)\) and we push-forward along \(\pi\). In the latter case, the push-forward formula involves the fundamental class \([\Omega_{\lambda}(E_{\bullet})]\) in \(\G_{d}(E)\).
It remains to compare the expressions obtained in each way for some generators of the Chow group of \(\Omega_{\lambda}(E_{\bullet})\).
The only serious obstacle to compute \([\Omega_{\lambda}(E_{\bullet})]\in A^{\bullet}(\G_{d}(E))\) may be the difficulty to solve this system of equations.

In the present paper, we treat type \(A\) and we heavily rely on the strong duality theorem to simplify this last step.
We express the sought Schubert class as a certain linear combination
\[
  [\Omega_{\lambda}(E_{\bullet})]
  =
  \sum \coeff_{\beta}s_{\beta}(U),
\]
with unknowns \(\coeff_{\beta}\), and we consider the push-forwards of the Schur classes \(s_{\alpha}(U)\) along \(\varpi_{\lambda}\).
Using the duality theorem to express the intersection of two Schur classes leads to an invertible triangular system in the unknowns \(\coeff_{\beta}\).
Solving this system, we get an expression for \([\Omega_{\lambda}(E_{\bullet})]\) in the Schur basis involving only bundles from the reference flag \(E_{\bullet}\).
In our computations we use extensively the algebra and combinatorics of (flagged) skew Schur functions.

\subsection{Push-forward formula for Schubert bundles}
Firstly, we derive push-forward formulas for Schubert bundles from the formulas for Kempf--Laksov bundles.
\begin{prop}
  \label{coro:pi*f}
  The push-forward along \(\varpi_{\lambda}\colon\Omega_{\lambda}(E_{\bullet})\to X\) of a symmetric polynomial \(f\) in \(d\) variables with coefficients in \(A^{\bullet}X\) is
  \[
    (\varpi_{\lambda})_{\ast}(f(U))
    =
    \Big[\tprod_{i=1}^{d} t_{i}^{\nu_{i}-1}\Big]
    \left(
    f(t_{1},\dotsc,t_{d})
    \tprod_{1\leq i<j\leq d}(t_{i}-t_{j})
    \tprod_{1\leq i\leq d}s_{1/t_{i}}(E_{\nu_{i}})
    \right).
  \]
\end{prop}
\begin{proof}
  Using the notation of Figure~\ref{fig:new}:
  \begin{align*}
    (\vartheta_{\nu})_{\ast}
    f\big(\varphi^{\ast}U\vert_{\Omega_{\lambda}(E_{\bullet})}\big)
    &=
    (\varpi_{\lambda})_{\ast}
    \varphi_{\ast}
    f\big(\varphi^{\ast}U\vert_{\Omega_{\lambda}(E_{\bullet})}\big)
    \\&=
    (\varpi_{\lambda})_{\ast}
    \big(
    f(U\vert_{\Omega_{\lambda}(E_{\bullet})})\cap
    \varphi_{\ast}[F_{\nu}(E_{\bullet})]
    \big)
    \\&=
    (\varpi_{\lambda})_{\ast}
    \big(
    f(U\vert_{\Omega_{\lambda}(E_{\bullet})})\cap
    [\Omega_{\lambda}(E_{\bullet})]
    \big).
  \end{align*}
  The first equality holds by commutativity, the second by the projection formula,
  and the third because \(\varphi\) is a desingularization.
  Thus, dropping the pullback notation
  \[
    (\vartheta_{\nu})_{\ast}f(U)
    =
    (\varpi_{\lambda})_{\ast}(f(U)\cap[\Omega_{\lambda}(E_{\bullet})]).
  \]
  For the left hand side, the formula
  \[
    (\vartheta_{\nu})_{\ast}f(U)
    =
    \Big[\tprod_{i=1}^{d} t_{i}^{\nu_{i}-1}\Big]
    \left(
    f(t_{1},\dotsc,t_{d})
    \tprod_{1\leq i<j\leq d}(t_{i}-t_{j})
    \tprod_{1\leq i\leq d}s_{1/t_{i}}(E_{\nu_{i}})
    \right)
  \]
  holds by Theorem~\ref{theo:gysin} for \(\mu=\nu\).
\end{proof}

\begin{prop}
  \label{prop:schur_omega}
  The push-forward along \(\varpi_{\lambda}\colon\Omega_{\lambda}(E_{\bullet})\to X\) of a Schur class \(s_{\alpha}(U)\), for a partition \(\alpha\) is
  \[
    (\varpi_{\lambda})_{\ast}(s_{\alpha}(U))
    =
    \det\big(
    s_{\alpha_{i}-i+j-\dual\lambda_{j}}(E_{\nu_{j}})
    \big)_{1\leq i,j\leq d}.
  \]
\end{prop}
\begin{proof}
  This follows from the previous considerations, since \(\nu_{j}=\dual\lambda_{j}+d+1-j\).
\end{proof}

\subsection{Schubert classes}
Lastly, we implement our strategy.
Before going on, let us define the following determinantal classes.
\begin{defi}[Flagged skew Schur classes]
  For \(\lambda,\gamma\) partitions with at most \(d\) parts and \(A_{\bullet},B_{\bullet}\) partial flags with \(d\) members, define
  \[
  s_{\lambda/\gamma}(A_{\bullet},B_{\bullet})
  \bydef
  \det\left(
  s_{\lambda_{i}-i+j-\gamma_{j}}(A_{i}+B_{j})
  \right)_{1\leq i,j\leq d}.
  \]
For shortness, we also define:
  \(
  s_{\lambda/\gamma}(A_{\bullet})
  \bydef
  s_{\lambda/\gamma}(A_{\bullet},0)
  \),
  and
  \(
  s_{\lambda/\gamma}^{*}(A_{\bullet})
  \bydef
  s_{\lambda/\gamma}(-A_{\bullet},0)
  \).
\end{defi}

We will need a variation for flagged Schur classes of the classical formula:
\[
  s_{\skew{\lambda}{\mu}}(A+B)
  =
  \sum_{\mu\subseteq\nu\subseteq\lambda}
  s_{\skew{\lambda}{\nu}}(A)
  s_{\skew{\nu}{\mu}}(B),
\]
(see~\cite{Macdonald}). Namely, we claim the following result.
\begin{lemm}
  \label{lemm:expansion}
  For \(\lambda,\gamma\) partitions with at most \(d\) parts and \(A_{\bullet},B_{\bullet}\) partial flags with \(d\) members, one has
  \[
  s_{\skew{\lambda}{\mu}}(A_{\bullet},B_{\bullet})
  =
  \sum_{\mu\subseteq\nu\subseteq\lambda}
  s_{\skew{\lambda}{\nu}}(A_{\bullet},0)\,
  s_{\skew{\nu}{\mu}}(0,B_{\bullet}).
  \]
\end{lemm}
The proof is direct, using the algorithm of straightening, so we admit it.

Let us denote by \(E/E_{\nu_{\bullet}}^{\rev}\) the dual flag of \(E_{\nu_{\bullet}}\), the \(j\)th member of which is the quotient bundle \(E/E_{\nu_{d+1-j}}\). Then, we can give the following compact formula for the Schubert classes in the Schur basis.
\begin{theo}
  With the notation of Section \ref{se:gysin}, for a partition \(\lambda\subseteq(n-d)^{d}\), the fundamental class of the Schubert bundle \(\Omega_{\lambda}(E_{\bullet})\subseteq\G_{d}(E)\) is
  \[
    [\Omega_{\lambda}(E_{\bullet})]
    =
    \sum_{\beta\subseteq\lambda}
  s_{\skew{\lambda}{\beta}}^{*}(E/E_{\nu_{\bullet}}^{\rev})\,
    s_{\beta}(U).
  \]
\end{theo}
\begin{proof}
  On one hand, the Schubert bundle \(\iota\colon\Omega_{\lambda}(E_{\bullet})\hookrightarrow\G_{d}(E)\) is a subvariety of \(\G_{d}(E)\), thus by the basis theorem its fundamental class \([\Omega_{\lambda}(E_{\bullet})]\in A^{\bullet}(\G_{d}(E))\) admits an expression under the form
  \[
  [\Omega_{\lambda}(E_{\bullet})]
  =
  \sum_{\beta\subseteq(n-d)^{d}}
  \coeff_{\beta}
  s_{\beta}(U).
  \]
  The push-forward of \(s_{\alpha}(U)\) along \(\iota\) is then
  \[
  \iota_{\ast}s_{\alpha}(U)
  =
  [\Omega_{\lambda}(E_{\bullet})]
  s_{\alpha}(U)
  =
  \sum_{\beta\subseteq(n-d)^{d}}
  \coeff_{\beta}
  s_{\alpha}(U)
  s_{\beta}(U).
  \]
  Now, by Theorem~\ref{theo:duality} the push-forward to \(X\) of \(s_{\alpha}(U)\) is
  \[
  (\varpi_{\lambda})_{\ast}s_{\alpha}(U)
  =
  (\pi\circ\iota)_{\ast}s_{\alpha}(U)
  =
  \sum_{\beta\subseteq(n-d)^{d}}
  \coeff_{\beta}
  \pi_{\ast}s_{\alpha}(U)s_{\beta}(U)
  =
  \sum_{\beta\subseteq(n-d)^{d}}
  \coeff_{\beta}
  s_{\skew{\alpha}{\dual\beta}}(E).
  \]

  On the other hand, by Proposition~\ref{prop:schur_omega} the push-forward to \(X\) of \(s_{\alpha}(U)\) is
  \[
  (\varpi_{\lambda})_{\ast}s_{\alpha}(U)
  =
  s_{\skew{\alpha}{\dual\lambda}}(0,E_{\nu_{\bullet}}).
  \]
  So, for each partition \(\alpha\), we get the equation in the unknowns \(\coeff_{\beta}\)
  \[
    \label{eq:alpha}
    \tag{\ensuremath{\alpha}}
    \sum_{\beta\subseteq(n-d)^{d}}
    \coeff_{\beta}
    s_{\skew{\alpha}{\dual\beta}}(E)
    =
s_{\skew{\alpha}{\dual\lambda}}(E,E_{\nu_{\bullet}}-E).
  \]
  Notice that the system made of equations \eqref{eq:alpha} for \(\alpha\subseteq(n-d)^{d}\) ordered by lexicographic order is a lower triangular system, with \(1\)'s on the diagonal.
  It is thus invertible and has a unique solution.
  Indeed,
  if \(\beta=\dual\alpha\), then \(s_{\skew{\alpha}{\dual\beta}}(E)=1\) and
  if \(\dual\beta\not\subseteq\alpha\), \textit{e.g.} if \(\beta\succ\dual\alpha\) in lexicographic order, then \(s_{\skew{\beta}{\dual\alpha}}(E)=0\).

  We will now modify the right hand side of \eqref{eq:alpha} to make it look alike the left hand side, using Lemma~\ref{lemm:expansion}.
   The equation \eqref{eq:alpha} becomes
   \[
   \sum_{\beta\subseteq(n-d)^{d}}
   \coeff_{\beta}
   s_{\skew{\alpha}{\dual\beta}}(E)
   =
   \sum_{\dual\lambda\subseteq\dual\beta\subseteq\alpha}
   s_{\skew{\alpha}{\dual\beta}}(E,0)
   s_{\skew{\dual\beta}{\dual\lambda}}(0,E_{\nu_{\bullet}}-E)
   =
   \sum_{\dual\lambda\subseteq\dual\beta\subseteq\alpha}
   s_{\skew{\alpha}{\dual\beta}}(E)
   s_{\skew{\lambda}{\beta}}^{*}(E-E_{\nu_{\bullet}}^{\rev}).
   \]
  Thus we obtain a solution (the unique solution by our above observations) by setting
  for any \(\beta\subseteq(n-d)^{d}\)
  \[
    \coeff_{\beta}
    =
    s_{\skew{\lambda}{\beta}}^{*}(E/E_{\nu_{\bullet}}^{\rev}).
    \qedhere
  \]
\end{proof}

For the sake of completeness, we now give a derivation of the Giambelli formula for vector bundles.
\begin{coro}[Giambelli formula]
  With the notation of Section \ref{se:gysin}, for a partition \(\lambda\subseteq(n-d)^{d}\), the fundamental class of the Schubert bundle \(\Omega_{\lambda}(E_{\bullet})\subseteq\G_{d}(E)\) is
  \[
    [\Omega_{\lambda}(E_{\bullet})]
    =
    \det(c_{\lambda_{i}-i+j}(E-E_{\nu_{d+1-i}}-U))_{1\leq i,j\leq d}.
  \]
\end{coro}
\begin{proof}
  This is again a plain application of Lemma~\ref{lemm:expansion}.
  One has
  \[
  s_{\lambda}(E_{\nu_{\bullet}}^{\rev}-E,U)
  =
  \sum_{\beta}
  s_{\skew{\lambda}{\beta}}(E_{\nu_{\bullet}}^{\rev}-E,0)
  s_{\beta}(0,U).
  \qedhere
  \]
\end{proof}

\subsection*{Acknowledgements}
Piotr Pragacz is supported by National Science Center (NCN) grant no.~2014/13/B/ST1/00133

\backmatter
\bibliographystyle{smfalpha}
\bibliography{gysin}
\vfill
\end{document}